\documentclass{amsart}

\usepackage{graphicx}
\usepackage[dvipsnames]{xcolor}

\newtheorem{theorem}{Theorem}[section]
\newtheorem{lemma}[theorem]{Lemma}
\newtheorem{proposition}[theorem]{Proposition}
\newtheorem{corollary}[theorem]{Corollary}

\newtheorem{mainthm}{Theorem}

\theoremstyle{definition}
\newtheorem{definition}[theorem]{Definition}

\theoremstyle{remark}
\newtheorem{remark}[theorem]{Remark}

\numberwithin{equation}{section}

\newcommand{\R}{\mathbb{R}}
\newcommand{\EE}{\mathbb{E}}
\newcommand{\set}[2]{\left\{ #1 \ | \ #2\right\}}
\newcommand{\vG}{\mathcal{G}}
\newcommand{\seg}{\underline}
\newcommand{\dd}{\mathrm{d}}

\DeclareMathOperator{\vol}{vol}
\DeclareMathOperator{\MV}{V}
\DeclareMathOperator{\erf}{erf}

\title{Gaussian Zonoids, Gaussian determinants and Gaussian random fields}

\author{Léo Mathis}
\address{Department of Mathematics, Goethe Universität, Frankfurt am Main, Germany}
\email{Mathis@mathematik.uni-frankfurt.de}


\begin{document}

\begin{abstract}
  We study the Vitale zonoid (a convex body associated to a probability distribution) associated to a non--centered Gaussian vector. This defines a family of convex bodies, that contains and generalizes ellipsoids, which we call \emph{Gaussian zonoids}. We show that each Gaussian zonoid can be approximated by an ellipsoid that we compute explicitely. We use this result to give new estimates for the expectation of the absolute {value of the} determinant of a non--centered Gaussian matrix in terms of mixed volume of ellipsoids. Finally, exploiting a recent link between random fields and zonoids uncovered by Stecconi and the author, we apply our results to the study of the zero set of non--centered Gaussian random fields. We show how these can be approximated by a suitable centered Gaussian random field and give a quantitative asymptotic in the limit where the variance goes to zero.
\end{abstract}

\maketitle

\section{Introduction}
Given a random vector $X\in\R^m$ such that $\EE\|X\|<+\infty$, Richard A. Vitale introduced in~\cite{Vitale} a way to build a convex body $\EE\seg{X}\subset\R^m$, associated to the law of $X$, that can be thought as the \emph{average} of the random convex body $\tfrac{1}{2}[-X,X]$, see~\eqref{eq:vitalezon} below and also~\cite[Definition~2.3]{ZA}. 

The convex bodies in $\R^m$ that can be obtained as $\EE\seg{X}$ for some random vector $X\in\R^m$ are called \emph{zonoids} \cite[Theorem~3.1]{Vitale} and play an important role in convex geometry, see for example~\cite[Section~3.5, 5.3 and 6.4]{bible}. Moreover this construction with random vectors appears in recent works such as~\cite{MOLCH1, MOLCH2, ZA, popeRandzon} and has a variant, introduced by Karl Mosler, called the \emph{lift zonoid} with applications in statistics, see~\cite{MoslerLiftzon}. Consistently with~\cite{ZA}, we will call $\EE\seg{X}$ the \emph{Vitale zonoid} associated to the random vector $X$. 

In this paper we study the Vitale zonoid associated to a Gaussian vector (that we assume non--degenerate), we call those \emph{Gaussian zonoids}. A similar construction can be found in~\cite[Example~2.10]{MoslerLiftzon}. In the case where the Gaussian vector $X\in\R^m$ is centered, that is when $\EE X=0$, the Gaussian zonoid $\EE\seg{X}$ is an ellipsoid, see Proposition~\ref{prop:GauBall} below. Therefore, Gaussian zonoids can be seen as a generalization of ellipsoids.


The main result is the following, see Theorem~\ref{thm:incluchain} for the precise statement and the definition of $b_\infty$.

\begin{mainthm}\label{introthm:incluchain}
    There is an universal constant $b_\infty\sim 0.91$ such that for {any} Gaussian zonoid $\EE\seg{X}$ there is an ellipsoid $\mathcal{E}$ such that 
    \begin{equation}\label{eq:incluchainintro}
        b_\infty\mathcal{E}\subset \EE\seg{X}\subset \mathcal{E}.
    \end{equation}
\end{mainthm}


As an immediate consequence, we get upper and lower bounds for the volume of Gaussian zonoids, see Corollary~\ref{cor:volwithlambda}. Moreover, by computing the volume of a limit convex body of Gaussian zonoids, we obtain, in Proposition~\ref{prop:lowvolandlim}, a better lower bound for the volume of Gaussian zonoids as well as an asymptotic when the mean of the Gaussian vector goes to infinity.

Vitale shows in~\cite[Theorem~3.2]{Vitale} that, if $M\in\R^{m\times m}$ is a random matrix with iid columns distributed as $X\in\R^m$, then $\EE|\det(M)|=m!\vol_m(\EE\seg{X})$. In \cite{MolchanovWespi}, Molchanov and Wespi generalize this result to the case where $M$ has independent columns that are not identically distributed, $\EE|\det(M)|$ is then equal to the \emph{mixed volume} of the associated Vitale zonoids.

{We then apply this result to Theorem~\ref{introthm:incluchain}. As a direct consequence, we obtain an estimate of }the expectation of the absolute determinant of a matrix with independent non--centered Gaussian vector with the mixed volume of ellipsoids, see Theorem~\ref{thm:randet} for the precise and more general statement.
\begin{mainthm}
     Let $M\in\R^{m\times m}$ be a random matrix whose columns are independent non--centered Gaussian vectors, then there are ellipsoids $\mathcal{E}_1,\ldots,\mathcal{E}_m$ such that
    \begin{equation}
        (b_\infty)^m\alpha_{m,m}\MV(\mathcal{E}_1,\ldots,\mathcal{E}_m)\leq \EE|\det(M)|\leq \alpha_{m,m} \MV(\mathcal{E}_1,\ldots,\mathcal{E}_m),
    \end{equation}
    where $\MV$ denotes the mixed volume, $\alpha_{m,m}$ is a constant depending only on the dimension $m$.
\end{mainthm}

This is to be compared with the result of Zakhar Kabluchko and Dmitry Zaporozhets in \cite{KabluchandZap} which prove that the expected absolute random determinant of a matrix with \emph{centered} Gaussian independents columns is \emph{equal} to the mixed volume of ellipsoids. Thus, Theorem~\ref{thm:incluchain} and Theorem~\ref{thm:randet} can be interpreted as follows: for every \emph{non--centered} Gaussian vector $X\in\R^m$ there is a \emph{centered} Gaussian vector $Y\in\R^m$ such that, for random determinants, $X$ is ``bounded'' by $Y$ from above and by $b_\infty Y$ from below.

Finally, we apply our result to the study of zero sets of Gaussian random fields. A Gaussian random field (GRF) is a random function $X$ on a smooth manifold $M$ such that the evaluation at {tuples} of points is a Gaussian vector. We use the recent results of Michele Stecconi and the author \cite{ZonSec} to link this study with the Gaussian zonoids, see Proposition~\ref{prop:MatStekmain}. We consider the case where the GRF is given by $X_\tau=\varphi+\tau Y$, where $\varphi\in C^\infty(M)$ is a fixed function, $\tau>0$ and $Y$ is a non--degenerate centered GRF with constant variance{, i.e., $Y$ is a GRF such that for all $p\in M$, $\EE[Y(p)]=0$ and $\EE[Y(p)^2]=1$}. For small $\tau>0$, the zero set $Z_\tau:=X_\tau^{-1}(0)$ can be seen as a random perturbation of the hypersurface $Z_0:=\varphi^{-1}(0)$. 

We obtain the following, see Theorem~\ref{thm:centeredGRFestimate} for the precise statement.
\begin{mainthm}
     For every $\tau>0$, there is a \emph{centered} GRF $\widetilde{X}_\tau$ with zero set $\widetilde{Z}_\tau:=\widetilde{X}_\tau^{-1}(0)$, such that for every \emph{admissible} random curve $\mathcal{W}$ and every open set $U$, we have
    \begin{equation} \label{eq:thmC}
    b_\infty\cdot\EE\#\left(\widetilde{Z}_\tau\cap \mathcal{W}\cap U\right)\leq\EE\#\left(Z_\tau\cap \mathcal{W}\cap U\right)\leq \EE\#\left(\widetilde{Z}_\tau\cap \mathcal{W}\cap U\right).
    \end{equation}
\end{mainthm}

Here, a random curve $\mathcal{W}$ is \emph{admissible} if it is the zero set of a vector valued non--degenerate GRF independent of $\widetilde{X}_\tau$ and ${X}_\tau$, see Remark~\ref{rk:vectGRF}. {Note that the expectations in \eqref{eq:thmC} can take value $+\infty$. In that case, \eqref{eq:thmC} implies that $\EE\#\left(Z_\tau\cap \mathcal{W}\cap U\right)=+\infty$ if and only if $\EE\#\left(\widetilde{Z}_\tau\cap \mathcal{W}\cap U\right)=+\infty$.}

We conclude this paper by studying the asymptotic behaviour of the random hypersurface $Z_\tau$ as $\tau\to 0$. In this limit, we expect $Z_\tau$ to \emph{concentrate} near $Z_0$. To make this statement rigorous and to describe \emph{how fast} this concentration occurs, we consider the neighborhood of $Z_0$ given by $\mathcal{U}_r:=\left\{|\varphi|<r\right\}$ and study the number
\begin{equation}
n_{r,\tau}:=\EE\#\left(Z^{(1)}_\tau\cap\cdots\cap Z^{(m)}_\tau\cap \mathcal{U}_r\right),
\end{equation}
where $Z^{(1)}_\tau,\ldots, Z^{(m)}_\tau$ are iid copies of $Z_\tau$.

We prove the following, see Theorem~\ref{thm:concentration} for the precise statement {and the explicit values of the constants}.
\begin{mainthm}
    Assume that $M$ is compact and that $0$ is a regular value of $\varphi$ and let $r=r(\tau)$ go to zero as $\tau\to 0$, we have:
    \begin{equation}\label{eq:limnrtintro}
    \lim_{\tau\to 0} n_{r,\tau}=c_m\cdot \erf\left(c'_m\cdot \alpha\right)\cdot \vol_{m-1}(Z_0),
    \end{equation}
    where $c_m,c'_m$ are constants depending only on the dimension $m$, $\alpha\in[0,+\infty]$ is the limit of $r/\tau$ as $\tau\to 0$, $\erf$ denotes the \emph{error function}, and the volume of $Z_0$ is the Riemannian volume in the Riemannian metric defined by $Y$ (see Definition~\ref{def:nondegGRF}).
\end{mainthm} 

This result can be interpreted as follows. If we let $r=\alpha\cdot\tau^{s}$ where $\alpha,s>0$, then the limit~\eqref{eq:limnrtintro} is zero for $s>1$ and is positive and constant (i.e. independent of $s$) for $s<1$. Equation~\eqref{eq:limnrtintro} then describes what happens in the critical regime $s=1$. 



\section{Gaussian zonoids}\label{sec:zonoids}

\subsection{Convex geometry}\label{subsec:convex}
We start by recalling some classical facts from convex geometry, the reader can refer to~\cite{bible} for more details. A \emph{convex body} in $\R^m$ is a non-empty convex compact subset of $\R^m$. If $K\subset \R^m$ is a convex body then its \emph{support function}, denoted $h_K:\R^m\to\R$, is defined for all $u\in\R^m$ by
\begin{equation}\label{eq:defsup}
    h_K(u):=\sup\set{\langle u, x\rangle}{x\in K}.
\end{equation}
It turns out that the support function characterizes the convex body, meaning that $K_1=K_2$ if and only if $h_{K_1}=h_{K_2}.$ Moreover the supremum norm of the support function (restricted to the sphere) defines a distance on the space of convex bodies called the \emph{Hausdorff distance}: $\dd(K,L):=\sup\set{|h_K(u)-h_L(u)|}{\|u\|=1}$, see~\cite[Lemma 1.8.14]{bible}. In the following, we will always consider the space of convex bodies endowed with the topology induced by this distance. We will denote by $B_m\subset\R^m$ the unit ball centered at $0$ and write $\kappa_m:=\vol_m(B_m)$.

The support function satisfy some properties that we summarize in the following proposition. These results being classical in convex geometry, we will omit the proof when it is not obvious but we will rather give the precise reference in \cite{bible}.

\begin{proposition}[Properties of the support function]\label{prop:supprop}
Let $K,L\subset \R^m$ be convex bodies. The following hold.
\begin{enumerate}
    \item  $K\subset L$ if and only if $h_K\leq h_L$.
    \item  If $T:\R^m\to\R^n$ is a linear map, then $h_{T(K)}=h_K\circ T^t.$
    \item Let $(K_n)_{n\in\mathbb{N}}$ be a sequence of convex bodies and let $h:\R^n\to\R$ be such that $h_{K_n}$ converges pointwise to $h$. Then $h$ is the support function of a convex body $K$ and $K_n$ converges to $K$ in the Hausdorff distance topology.
    \item If $h_K$ is differentiable at $u\in S^{m-1}$, then its gradient $\nabla h_K(u)$ is the unique point of the boundary $\partial K$ that admits $u$ as an outer unit normal. In particular if $h_K$ is $C^1$ on $\R^m\setminus\{0\}$ then the restriction of the gradient on the unit sphere $(\nabla h_K)_{|S^{m-1}}$ parametrises the boundary $\partial K$.
\end{enumerate}
\end{proposition}

\begin{proof}
Item 1 follows from the fact that $x\in K$ if and only if $\langle u, x \rangle \leq h_K(u)$ for all $u\in \R^m$. Item 2 is a direct consequence of the definition of the support function in~\eqref{eq:defsup}. Finally, item 3 is~\cite[Theorem~1.8.15]{bible} and item 4 is~\cite[Corollary~1.7.3]{bible}.
\end{proof}

Let us recall the notion of \emph{mixed volume}. Given two convex bodies $K,L\subset \R^m$ one can define their \emph{Minkoxski sum}: $K+L:=\set{x+y}{x\in K,\ y\in L}.$ A fundamental result by Minkowski (see~\cite[Theorem~5.1.7]{bible}) says that given convex bodies $K_1,\ldots,K_m\subset \R^m$ the function $(t_1,\ldots,t_m)\mapsto\vol_m(t_1K_1+\cdots+t_mK_m)$ is a polynomial in $t_1,\ldots,t_m\in\R$. The coefficient of $t_1\cdots t_m$ is called the \emph{mixed volume} of $K_1,\ldots,K_m$ and is denoted $\MV(K_1,\ldots,K_m)$. In some sense this is a \emph{polarization} of the function $\vol_m$ on convex bodies of $\R^m$. We gather the properties of mixed volume that will be useful for us in the next proposition. Once again, we omit the proof but give the precise reference in~\cite{bible}.

\begin{proposition}[Properties of the mixed volume]\label{prop:MVprop}
Let $K,L,K_1,\ldots,K_m\subset \R^m$ be convex bodies  and $a,b\geq0$. The following hold.
\begin{enumerate}
\item The mixed volume is completely symmetric: for any permutation $\sigma\in\mathfrak{S}_m$, we have $\MV(K_{\sigma(1)},\ldots, K_{\sigma(m)})=\MV(K_1,\ldots,K_m)$.
\item The mixed volume is Minkowski linear in each variable: $\MV(aK+bL,K_2,\ldots,K_m)=a\MV(K,K_2,\ldots,K_m)+b\MV(L,K_2,\ldots,K_m)$.
\item We have $\MV(K,\ldots,K)=\vol_m(K)$. More generally, if $K$ is of dimension $k\leq m$, we have 
$\vol_k(K)=\frac{\binom{m}{k}}{\kappa_{m-k}}\MV(K[k],B_m[m-k]),$
where recall that $B_m\subset\R^m$ is the unit ball, $\kappa_{j}:=\vol_j(B_j)$ and where $L[j]$ denotes that the convex body $L$ is repeated $j$ times in the argument of the mixed volume.
\item We have $\MV(K_1,\ldots,K_m)>0$ if and only if there exist segments $[x_1,y_1]\subset K_1$, $\ldots,[x_m,y_m]\subset K_m$ such that $y_1-x_1,\ldots y_m-x_m$ is a basis of $\R^m$. Otherwise $\MV(K_1,\ldots,K_m)=0$.
\item The mixed volume is monotone: if $K\subset L$ then $\MV(K,K_2,\ldots,K_m)\leq\MV(L,K_2,\ldots,K_m)$.
\end{enumerate}
\end{proposition}

\begin{proof}
Item 1 is \cite[Theorem~5.1.7]{bible}; item 2 is \cite[(5.26)]{bible}; item 3 is a special case of \cite[(5.31)]{bible}; item 4 is \cite[Theorem~5.1.8]{bible} and item 5 is \cite[(5.25)]{bible}.
\end{proof}

We say that a random vector $X\in\R^m$ is \emph{integrable} if $\EE\|X\|<\infty$. Given an integrable random vector $X\in\R^m$, we denote by $\EE\seg{X}$ the convex body with support function 
\begin{equation}\label{eq:vitalezon}
    h_{\EE\seg{X}}(u):=\frac{1}{2}\EE|\langle u, X\rangle|.
\end{equation}
This function is the support function of a convex body because it is sublinear, see~\cite[Theorem~1.7.1]{bible}. The convex bodies that can be obtained {in} this way (and their translates) are called \emph{zonoids}, see~\cite[Theorem~3.1]{Vitale} and the zonoid $\EE\seg{X}$ is called the \emph{Vitale zonoid} associated to the random vector $X$ (see~\cite{ZA} where $\EE\seg{X}$ is denoted $K(X)$). Note that $h_{[-X,X]}(u)=|\langle u, X\rangle|$ and thus $\EE\seg{X}$ can be thought of as the \emph{expectation} of the random convex body 
\begin{equation}
	\seg{X}:=\tfrac{1}{2}[-X,X]
\end{equation} 
For more details on the theory of random sets, the reader can refer to~\cite{molchRandSets} where $\EE\seg{X}$ is a particular case of a more general notion called the \emph{selection expectation}.

The following observation is~\cite[Proposition~2.4]{ZA}.
\begin{lemma}\label{lem:vitzonlin}
Let $X\in\R^m$ be an integrable random vector and let $M:\R^m\to \R^n$ be a linear map. Then $M(X)\in\R^n$ is integrable and we have $\EE\seg{M(X)}={M}(\EE\seg{X}).$
\end{lemma}

\subsection{Gaussian vectors and their zonoids}
In this paper, we are interested in a particular class of integrable random vectors, namely the \emph{Gaussian vectors}. Recall that a random vector $X\in\R^m$ is called \emph{Gaussian} if for every $u\in\R^m$, the random variable $\langle u,X\rangle$ is a Gaussian variable. In that case the distribution of $X$ is determined by its mean $\EE X$ and its \emph{covariance matrix} $\EE\left[(X-\EE X)(X-\EE X)^t\right]\in\R^{m \times m}$. 
In the following, we assume that all Gaussian vectors are \emph{non--degenerate}, i.e. that their support is the whole space $\R^m$.

\begin{definition}
A convex body $K\subset\R^m$ is called a \emph{Gaussian zonoid} if there is a Gaussian vector $X\in\R^m$ such that $K=\EE\seg{X}$.
\end{definition}

A particular case is the \emph{standard Gaussian vector} $\xi\in\R^m$ which has mean $0$ and whose covariance matrix is the identity matrix. Its density is given for all $x\in\R^m$ by $\rho(x)=(2\pi)^{-\frac{m}{2}}\exp\left(-\frac{\|x\|^2}{2}\right)$. 
For every (non--degenerate) Gaussian vector $X\in\R^m$, there is a linear map $T:\R^m\to\R^m$ and a constant $s\geq 0$ such that $X$ is distributed as $T(s \, e_1+\xi)$ where $e_1,\ldots, e_m$ is the standard basis of $\R^m$.
We use this fact and Lemma~\ref{lem:vitzonlin} to reduce our study to the case where the Gaussian vector is of the form $s \, e_1+\xi$, i.e. has covariance matrix the identity, hence the following definition.

\begin{definition}\label{def:G(c)}
For every $s\geq 0$, we define 
\begin{equation}
    G(s):=\EE\seg{s e_1+\xi},
\end{equation}
where $\xi\in\R^m$ is a standard Gaussian vector.
\end{definition}

It follows from Lemma~\ref{lem:vitzonlin} that a convex body $K\subset\R^m$ is a Gaussian zonoid if and only if there exists a linear map $T:\R^m\to\R^m$ and a vector $v\in\R^m$ such that $K=T(G(v)).$

In what follows, we will make extensive use of the \emph{error function} $\erf:\R\to\R$ that is given for every $t\in\R$ by
\begin{equation}\label{eq:erfdef}
\erf(t):=\frac{2}{\sqrt{\pi}}\int_0^t e^{-s^2}\dd s.
\end{equation}

Moreover, we will need the following result which is~\cite[(3)]{foldednormal} {(one can also find this formula in the Wikipedia article~\cite{wiki:foldedNormal})}.

\begin{lemma}\label{lem:foldgauss}
Let $\gamma\in\R$ be a Gaussian variable of mean $\mu\in\R$ and variance $\sigma^2>0$. Then its first absolute moment is given by
\begin{equation}
    \EE|\gamma|=\sigma {\sqrt {\frac {2}{\pi }}}\,\,\exp \left({\frac {-\mu ^{2}}{2\sigma ^{2}}}\right)+\mu \,{\erf}\left({\frac {\mu }{\sqrt {2\sigma ^{2}}}}\right).
\end{equation}

\end{lemma}

The first example is when $s=0$.

\begin{proposition}\label{prop:GauBall}
If $\xi\in\R^m$ is a standard Gaussian vector then we have
\begin{equation}
    G(0)=\EE\seg{\xi}=\frac{1}{\sqrt{2\pi}}B_m,
\end{equation}
where $B_m$ is the unit ball of $\R^m$.
\end{proposition}
\begin{proof}
Since the random vector $\xi$ is invariant under the action of $O(m)$, by Lemma~\ref{lem:vitzonlin}, $\EE\seg{\xi}$ must be a ball. It is then enough to compute $h_{\EE\seg{\xi}}(e_1)=\frac{1}{2}\EE|\langle e_1, \xi\rangle|.$
The random variable $\langle e_1, \xi\rangle$ is a standard Gaussian variable and thus, applying Lemma~\ref{lem:foldgauss} with $\mu=0$ and $\sigma=1$, we get  $\EE|\langle e_1, \xi\rangle|=\sqrt{\tfrac{2}{\pi}}$ which concludes the proof.
\end{proof}

In general, we can compute the support function of $G(s)$ explicitly, see Figure~\ref{fig:GaussianEye}. Note that, by Lemma~\ref{lem:vitzonlin}, $G(s)$ is invariant under the action of $O(m-1)$, the stabilizer of $e_1$ in the orthogonal group $O(m)$, i.e. $G(s)$ is a solid of revolution around the axis spanned by $e_1\in\R^m.$


In what follows, we will write a vector $u\in \R^m=\R e_1\times\R^{m-1}$ as 
\begin{equation}\label{eq:defxy}
    u=(x,y) \quad \text{where} \quad x:=\langle u, e_1 \rangle \in\R \quad \text{and}\quad y\in e_1^\perp=\R^{m-1}
\end{equation}

\begin{proposition}[Support function of Gaussian zonoids]\label{prop:suppG}
Let $s\geq 0$. The support function of $G(s)$ is given, for all $(x,y)\in\R\times\R^{m-1}=\R^m$, in the notation introduced in \eqref{eq:defxy}, by
\begin{equation}
    h_{G(s)}(x,y)=\frac{\sqrt{x^2+\|y\|^2}}{\sqrt{2\pi}}\exp\left(\frac{-s^2x^2}{2(x^2+\|y\|^2)}\right)+\frac{sx}{2}\erf\left(\frac{sx}{\sqrt{2}\sqrt{x^2+\|y\|^2}}\right),
\end{equation}
where recall the definition of the error function $\erf$ in \eqref{eq:erfdef}.
\end{proposition}
\begin{proof}
Consider the Gaussian random variable $\zeta:=\langle (x,y), s e_1+\xi\rangle=xs+\langle (x,y), \xi\rangle\in\R$. We can compute its mean $\EE\zeta=xs$ and variance $\EE\left[(\zeta-sx)^2\right]=\EE\left[\langle (x,y),\xi\rangle^2\right]=x^2+\|y\|^2$. By definition of the zonoid $G(s)$ (Definition~\ref{def:G(c)}) and by \eqref{eq:vitalezon}, the support function of $G(s)$ is given by $h_{G(s)}(x,y)=\tfrac{1}{2}\EE|\zeta|$. Applying then Lemma~\ref{lem:foldgauss} for $\mu=sx$ and $\sigma^2=x^2+\|y\|^2$ and dividing by $2$ gives the result.
\end{proof}
\begin{figure}
    \centering
    \includegraphics[scale=0.8]{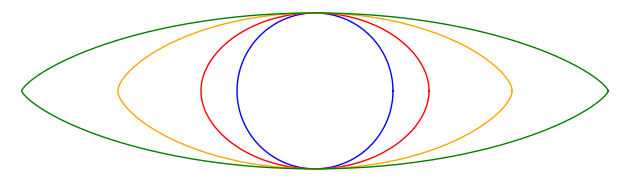}
    \caption{The Gaussian eye: the zonoids $G(s)$ for $s=0,1,2$ and $3$.}
    \label{fig:GaussianEye}
\end{figure}

Let us denote by $\mathcal{K}^m$ the space of convex bodies with the topology induced by the Hausdorff distance (see Section~\ref{subsec:convex} above).

\begin{proposition}
The map $G:\R_{\geq0}\to\mathcal{K}^m$, $s\mapsto G(s)$ is continuous and strictly increasing with respect to inclusion on $s>0.$
\end{proposition}
\begin{proof}
Let us first prove continuity. By Proposition~\ref{prop:suppG}, the function  $h_{G(\cdot)}(\cdot):\R_{\geq0}\times\R^m\to \R$ given by $(s,u)\mapsto h_{G(s)}(u)$ is continuous. In particular, if $s_n\in\R_{\geq0}$ is a sequence that converges to $s\in\R_{\geq0}$, for all $u\in\R^m$, $h_{G(s_n)}(u)\to h_{G(s)}(u)$. In other words, $h_{G(s_n)}$ converges pointwise to $h_{G(s)}$ and we conclude by Proposition~\ref{prop:supprop}--(3).

Now let us prove that it is increasing. By Proposition~\ref{prop:supprop}--(1), it is enough to show that given a fixed non zero point $(x,y)\in\R\times\R^{m-1}$, the function $s\mapsto h_{G(s)}(x,y)$ is increasing. From Proposition~\ref{prop:suppG}, we get:
\begin{equation}
    \frac{\dd}{\dd s} h_{G(s)}(x,y)= \frac{x}{2}\erf\left(\frac{s x}{\sqrt{x^2+\|y\|^2}}\right)
\end{equation}
which is non--negative for all $s\geq 0$ which proves the monotonicity. To prove that it is \emph{strictly} increasing, it is enough to note that $\frac{\dd}{\dd s} h_{G(s)}(x,y)>0$ if $x\neq0$ and $s>0$ and this concludes the proof.
\end{proof}

Note that this propositions implies that the map $v\mapsto \EE\seg{v+\xi}$ where $v\in\R^m$ is a vector, is also continuous.

We see, from Proposition~\ref{prop:suppG}, that for $s>0$ the Gaussian zonoid $G(s)$ is not an ellipsoid. However we shall show, in Theorem~\ref{thm:incluchain} below, that it remains close to one. In order to state this result, let us first introduce a few definitions.

We define $\lambda:\R\to\R$ to be given for all $s\in\R$ by
\begin{equation}\label{eq:lambdac}
    \lambda(s):= e^{\frac{-s^2}{2}}+\sqrt{\frac{\pi}{2}}s\erf\left(\frac{s}{\sqrt{2}}\right).
\end{equation}
Note that, by Proposition~\ref{prop:suppG} and following the same notation, for all $s\geq 0$, we have $h_{G(s)}(1,0)=\frac{\lambda(s)}{\sqrt{2\pi}}.$ Moreover, one can see that we have $\lambda'(s)=\sqrt{\tfrac{\pi}{2}}\erf\left(\tfrac{s}{\sqrt{2}}\right).$
Then we define the function $\varphi_\infty : \R^2\to \R$ given for all $(x,z)\in\R^2$ by
\begin{equation}\label{eq:phifty}
    \varphi_\infty(x,z):=|z| e^{\frac{-x^2}{\pi z^2}}+x\erf\left(\frac{x}{\sqrt{\pi}|z|}\right).
\end{equation}

Finally, we define the linear map $T_s:\R^m\to \R^m$, given, in the standard basis, by the diagonal matrix
\begin{equation}\label{def:Ts}
T_s:=\left( \begin{matrix} \lambda(s) 	    & 	& 0	   &	\\
											&1 	&	   &	\\
								0			&	&\ddots&    \\
											&	&	   &  1
\end{matrix}\right)
\end{equation}
In other words, $T_s(x,y)=(\lambda(s)x,y)$. Note that, since $\lambda(0)=1$, the map $T_0$ is the identity.

\begin{theorem}[Inclusion chain of ellipsoids and Gaussian zonoids]\label{thm:incluchain}
For all $s\geq 0$, we have:
\begin{equation}
    b_\infty T_s\left(\frac{1}{\sqrt{2\pi}}B_m\right)\subset G(s) \subset T_s\left(\frac{1}{\sqrt{2\pi}}B_m\right),
\end{equation}
where $B_m\subset \R^m$ is the unit ball and $b_\infty:=\min \set{\varphi_\infty(\cos(t),\sin(t))}{t\in[0,2\pi]}\sim 0.91\ldots$
\end{theorem}
\begin{proof}
If $s=0$, $G(0)$ is equal to the upper bound and there is nothing to prove. Thus we can assume without loss of generality that $s>0$. Let $\widetilde{G}(s):=\sqrt{2\pi}T_{s}^{-1}G(s)$. The idea of the proof is to show that the map {$s\mapsto \widetilde{G}(s)$} is strictly decreasing with respect to inclusion for $s>0$. Once this is established, it is enough to show that the limit object $\widetilde{G}(\infty)$ exists and contains a ball of radius $b_\infty$.

Let us first show that $s\mapsto \widetilde{G}(s)$ is decreasing. We compute the support function $h_{\widetilde{G}(s)}(x,y)=\sqrt{2\pi}h_{G(s)}\left(\frac{x}{\lambda(s)},y\right)$, where we used the notation introduced in \eqref{eq:defxy}. A straightforward computation shows that for $(x,y)\neq(0,0),$ we have:
\begin{equation}
    h_{\widetilde{G}(s)}(x,y)=\alpha(s)\lambda\left(\beta(s)\right),
\end{equation}
where 
\begin{align}
    \alpha(s)&:=\frac{\sqrt{x^2+\lambda^2(s)\|y\|^2}}{\lambda(s)}; & \beta(s)&:=\frac{xs}{\sqrt{x^2+\lambda^2(s)\|y\|^2}}.
\end{align}
Moreover, we note that $h_{\widetilde{G}(s)}(\pm x,0)=|x|$ and $h_{\widetilde{G}(s)}(0,y)=\|y\|$ do not depend on $s$. Thus we fix $x>0$ and $\|y\|\neq0$ and omit the dependence in $(x,y)$ writing $\varphi(s):=h_{\widetilde{G}(s)}(x,y)=\alpha(s)\lambda(\beta(s)).$ It is then enough to show that $\varphi(s)$ is decreasing on $s>0$, or equivalently that $\varphi'(s)<0$ $\forall s>0$. This requires a bit of work. First we compute $\alpha'$ and $\beta'$ finding:
\begin{align}
    \alpha'(s)&=\frac{-x\beta(s)\lambda'(s)}{\lambda^2(s)s}; & \alpha(s)\beta'(s)&=\frac{x}{\lambda(s)}-\frac{\beta^2(s)\|y\|^2\lambda'(s)}{xs}.
\end{align}
Introducing these in the expression $\varphi'(s)=\alpha'(s)\lambda(\beta(s))+\alpha(s)\beta'(s)\lambda'(\beta(s))$, we find the following.
\begin{equation}\label{eq:thiseqphipr}
    \frac{\lambda^2(s)s}{x\lambda'(s)\lambda'(\beta(s))}\varphi'(s)=-\beta(s)\frac{\lambda(\beta(s))}{\lambda'(\beta(s))}+s\frac{\lambda(s)}{\lambda'(s)}-\frac{\lambda^2(s)\beta^2(s)\|y\|^2}{x^2}.
\end{equation}
We will now express this in terms of yet another function:
\begin{equation}
    \rho(t):=t\frac{\erf'(t)}{\erf(t)}=\sqrt{\frac{2}{\pi}}\, \frac{ t e^{-t^2}}{\erf(t)}.
\end{equation}
We note that $t\frac{\lambda(t)}{\lambda'(t)}=\rho\left(\frac{t}{\sqrt{2}}\right)+t^2$. Thus, reintroducing in \eqref{eq:thiseqphipr} yields:
\begin{equation}\label{eq:tisothereqphir}
     \frac{\lambda^2(s)s}{x\lambda'(s)\lambda'(\beta(s))}\varphi'(s)=\rho\left(\frac{s}{\sqrt{2}}\right)-\rho\left(\frac{\beta(s)}{\sqrt{2}}\right)+s^2-\beta^2(s)-\frac{\lambda^2(s)\beta^2(s)\|y\|^2}{x^2}.
\end{equation}
It remains to see that $\beta^2(s)+\frac{\lambda^2(s)\beta^2(s)\|y\|^2}{x^2}=\frac{\beta^2(s)}{x^2}(x^2+\lambda^2(s)\|y\|^2)=s^2$. Thus, \eqref{eq:tisothereqphir} becomes:
\begin{equation}\label{eq:phirfinal}
     \frac{\lambda^2(s)s}{x\lambda'(s)\lambda'(\beta(s))}\varphi'(s)=\rho\left(\frac{s}{\sqrt{2}}\right)-\rho\left(\frac{\beta(s)}{\sqrt{2}}\right).
\end{equation}
Now, Horst Alzer shows in~\cite[Lemma~2.1]{erfineq} that $\rho(t)$ is strictly decreasing for $t>0$. Moreover since $x>0$ and $y\neq 0$, we have that $\beta(s)<s$ and thus $\rho\left(\frac{\beta(s)}{\sqrt{2}}\right)>\rho\left(\frac{s}{\sqrt{2}}\right)$ for all $s>0$.
Since the coefficient in front of $\varphi'$ in~\eqref{eq:phirfinal} is positive for all $s>0$, this shows that $\varphi'(s)<0$ $\forall s>0$. In definitive we have shown that for all $s>0$ the map $s\mapsto \widetilde{G}(s)$ is (strictly) decreasing with respect to inclusion.

We now proceed to study the limit as $s\to+\infty.$ For any $y\neq 0$, we have $\lim_{s\to+\infty}\alpha(s)=\|y\|$ and $\lim_{s\to+\infty}\beta(s)=\sqrt{\frac{2}{\pi}}\frac{x}{\|y\|}.$
We obtain, for any fixed $(x,y)\in\R\times\R^{m-1}$,
\begin{equation}
    \lim_{s\to+\infty}\varphi(s)=\varphi_\infty(x,\|y\|),
\end{equation}
where recall the definition of $\varphi_\infty$ in~\eqref{eq:phifty}. By Proposition~\ref{prop:supprop}--(3), this is the support function of a convex body (zonoid) that we denote $\widetilde{G}(\infty)$ and we have $\widetilde{G}(s)\to\widetilde{G}(\infty)$ as $s\to+\infty$. By what we just proved, for all $s>0$ we have 
\begin{equation}\label{eq:inclugifty}
    \widetilde{G}(\infty)\subset \widetilde{G}(s) \subset \widetilde{G}(0)=B_m.
\end{equation}
Moreover, $\widetilde{G}(\infty)$ contains a ball of radius 
\begin{equation}
    \min\set{h_{\widetilde{G}(\infty)}(u)}{\|u\|=1}=\min \set{\varphi_\infty(\cos(t),\sin(t))}{t\in[0,2\pi]}=:b_\infty.
\end{equation}
 Mapping everything through the linear map $\frac{1}{\sqrt{2\pi}}T_{s e_1}$ (which preserves inclusion) gives the result.
\end{proof}

\begin{remark}
    This proof is central to this paper. The proof that $s\to \widetilde{G}(s)$ is decreasing is, as the reader may have noticed, highly technical and rely on some non trivial properties of the error function. Most of the other results of this paper are based on this result. The author hope that this will convince that they are highly non trivial.
\end{remark}

From Theorem~\ref{thm:incluchain} and the fact that $\det(T_s)=\lambda(s)$, we get an estimate on the volume of the Gaussian zonoids $G(s)$.

\begin{corollary}[Volume bounds for $G(s)$]\label{cor:volwithlambda}
For every $s\geq0$ we have 
\begin{equation}
    (b_\infty)^m \frac{\lambda(s)}{(2\pi)^{\frac{m}{2}}}\kappa_m\leq \vol_m(G(s))\leq \frac{\lambda(s)}{(2\pi)^{\frac{m}{2}}}\kappa_m,
\end{equation}
where recall that $\kappa_m$ is the volume of the unit ball $B_m\subset \R^m$, $\lambda$ is defined by~\eqref{eq:lambdac} and $b_\infty$ is defined in Theorem~\ref{thm:incluchain}.
\end{corollary}

In the proof of Theorem~\ref{thm:incluchain} we proved that, for all $s\geq 0$, $G(s)$ contains the convex body $\tfrac{1}{\sqrt{2\pi}}T_s(\widetilde{G}(\infty))$, see~\eqref{eq:inclugifty}. Moreover we also proved that, as $s\to+\infty$, the renormalized convex body $(T_s)^{-1}(G(s))$ converges to $(1/\sqrt{2\pi})\widetilde{G}(\infty)$. Thus, we can improve the lower bound on the volume of $G(s)$ and compute an asymptotic by computing the volume of $\widetilde{G}(\infty)$.

\begin{proposition}[Lower bound and asymptotic of volume]\label{prop:lowvolandlim}
For all $s\geq 0$, we have 
\begin{equation}
	\frac{\lambda(s)}{(2\pi)^{\frac{m}{2}}}\cdot\frac{2\kappa_{m-1}}{\sqrt{m}}\leq \vol_m(G(s)).
\end{equation}
Moreover, as $s\to+\infty$, we have:
\begin{equation}
\lim_{s\to+\infty}\frac{1}{s}\vol_m(G(s))=\frac{1}{(2\pi)^{\frac{m-1}{2}}}\cdot\frac{\kappa_{m-1}}{\sqrt{m}}.
\end{equation}
\end{proposition}
\begin{proof}
We are going to show that 
\begin{equation}\label{eq:volGinfty}
vol_m(\widetilde{G}(\infty))=2\kappa_{m-1}/\sqrt{m}.
\end{equation}
 Then the first statement follows from the fact that $G(s)$ contains $\tfrac{1}{\sqrt{2\pi}}T_s(\widetilde{G}(\infty))$.

Recall that the support function of $\widetilde{G}(\infty)$ is given by $h_{\widetilde{G}(\infty)}(x,y)=\varphi_\infty(x,\|y\|)$, where $\varphi_\infty$ is defined in \eqref{eq:phifty}. We compute its gradient and find for $\|y\|\neq 0$:
$\frac{\partial}{\partial x} h_{\widetilde{G}(\infty)}(x,y)=\erf\left(\frac{x}{\sqrt{\pi}\|y\|}\right)$ and $\frac{\partial}{\partial y_i} 	h_{\widetilde{G}(\infty)}(x,y)	=\frac{y_i}{\|y\|}\exp\left(\frac{-x^2}{\pi\|y\|^2}\right)$

Thus, by Proposition~\ref{prop:supprop}-(4), the boundary of $\widetilde{G}(\infty)$ is described by the equation $\|y\|=f(x)$ with $f(x):=\exp\left(-(\erf^{-1}(x))^2\right)$ and $x\in[-1,1]$. Hence its volume is given by
\begin{equation}
\vol_m(\widetilde{G}(\infty))=\int_{-1}^{1}\kappa_{m-1} f(x)^{m-1} \dd x.
\end{equation}
We apply the change of variable $x=\erf(u)$ to find
\begin{equation}
\vol_m(\widetilde{G}(\infty))=\kappa_{m-1} \frac{2}{\sqrt{\pi}}\int_{-\infty}^{\infty} e^{-m u^2} \dd u.
\end{equation}
Changing variables again with $u=s/\sqrt{2m}$ gives~\eqref{eq:volGinfty} and thus the first part of the statement.

To prove the limit, we notice first that, since $(T_s)^{-1}(G(s))$ converges to $(1/\sqrt{2\pi})\widetilde{G}(\infty)$ as $s\to+\infty$, we get that 
\begin{equation}
\lim_{s\to+\infty}\frac{1}{\lambda(s)}\vol_m(G(s))=\frac{1}{(2\pi)^{\frac{m}{2}}}\cdot\frac{2\kappa_{m-1}}{\sqrt{m}}.
\end{equation}
Then it is enough to note that $\lim_{s\to+\infty} \frac{\lambda(s)}{s}=\sqrt{\frac{\pi}{2}}$.
\end{proof}

\section{Application to random determinants}
\label{sec:randet}


Richard Vitale shows in~\cite{Vitale} that if $T\in\R^{m\times m}$ is a random matrix with iid columns distributed as some integrable random vector $X\in\R^m$, then $\EE|\det(T)|=m!\vol_m(\EE\seg{X})$ (recall the definition of $\EE\seg{X}$ in~\eqref{eq:vitalezon}). In \cite{MolchanovWespi}, Ilya Molchanov and Florian Wespi, generalize this result to the case of independent columns (not necessarily identically distributed) and rectangular matrices. Lemma~\ref{lem:vitale} below is  \cite[Theorem~2.1]{MolchanovWespi} reformulated in a language more suitable for our context.

Note that \cite[Section~5]{ZA}, also investigates generalizations of Vitale's result and we will provide a proof of Lemma~\ref{lem:vitale} below based on~\cite[Theorem~5.4]{ZA}.


\begin{lemma}[Expected determinants and Vitale zonoids]\label{lem:vitale}
Let $0<k\leq m$ and let $X_1,\ldots,X_k\in\R^m$ be independent integrable random vectors. Consider the random matrix $\Gamma:=\left(X_1,\ldots,X_k\right)\in\R^{m\times k}$ whose columns are the vectors $X_i$ and let $K_i:=\EE\seg{X_i}$ be the zonoid defined in~\eqref{eq:vitalezon}. Then we have
\begin{equation}
    \EE\sqrt{\det(\Gamma^t\Gamma)}=\frac{m!}{(m-k)!\kappa_{m-k}}\MV(K_1,\ldots,K_k,B_m[m-k]),
\end{equation}
where $B_m[m-k]$ denotes the unit ball $B_m\subset\R^m$ repeated $m-k$ times in the argument and where recall that $\kappa_{m-k}=\vol_{m-k}(B_{m-k})$.
\end{lemma}
\begin{proof}
First note that $\sqrt{\det\left(\Gamma^t\Gamma\right)}$ is equal to the $k$--th dimensional volume of the parallelotope spanned by the vectors $X_1,\ldots,X_k$, i.e. to the volume of the Minkowski sum $[0,X_1]+\cdots+[0,X_k]$. Writing $\seg{X}$ to denote the segment $\frac{1}{2}[-X,X]$ (which is a translate of the segment $[0,X]$), we have that $\sqrt{\det\left(\Gamma^t\Gamma\right)}$ is equal to the $k$--th dimensional volume of $\seg{X_1}+\cdots+\seg{X_k}$. Using Proposition~\ref{prop:MVprop}--(3), we obtain
\begin{equation}
    \sqrt{\det\left(\Gamma^t\Gamma\right)}=\frac{\binom{m}{k}}{\kappa_{m-k}}\MV(\seg{X_1}+\cdots+\seg{X_k}[k],B_m[m-k])=\frac{m!}{(m-k)!\kappa_{m-k}}\MV(\seg{X_1},\ldots,\seg{X_k},B_m[m-k]),
\end{equation}
where the second equality is obtained by expanding by multilinearity the mixed volume (Proposition~\ref{prop:MVprop}-(2)) and eliminating every term where a segment appear twice by Proposition~\ref{prop:MVprop}-(4).
Finally by~\cite[Theorem~5.4]{ZA} we obtain, taking the expectation: 
\begin{equation}
\EE\MV(\seg{X_1},\ldots,\seg{X_k},B_m[m-k])=\MV(K_1,\ldots,K_k,B_m[m-k]),
\end{equation}
which concludes the proof.
\end{proof}

In the case where $X_i$ is a (non--degenerate) Gaussian vector, the zonoid $K_i$ is, by definition, a Gaussian zonoid and thus is a linear image of some $G(s_i)$ (see the previous section). Since the mixed volume is increasing with respect to inclusion (Proposition~\ref{prop:MVprop}-(5)), Lemma~\ref{lem:vitale} and Theorem~\ref{thm:incluchain} imply the following. 
\begin{theorem}[Gaussian determinant estimates]\label{thm:randet}
Let $0<k\leq m$ and let $X_1,\ldots,X_k\in\R^m$ be independent Gaussian vectors such that $X_i=M_i(s_i e_1+\xi_i)$ with $M_i:\R^m\to\R^m$ an invertible linear map, $s_i\geq0$ fixed constants and $\xi_i\in\R^m$ iid standard Gaussian vectors. Consider the random matrix $\Gamma:=(X_1,\ldots,X_k)$ whose columns are the vectors $X_i$ and define the ellipsoids $\mathcal{E}_i:=M_i\circ T_{s_i} \left(B_m\right)$ for $i=1,\ldots,k$, where recall the definition of $T_s$ in \eqref{def:Ts}. We have
\begin{equation}
    (b_\infty)^k \alpha_{m,k} \MV\left(\mathcal{E}_1,\ldots,\mathcal{E}_k,B_m[m-k]\right)\leq  \EE\sqrt{\det\left(\Gamma^t\Gamma\right)}\leq  \alpha_{m,k} \MV\left(\mathcal{E}_1,\ldots,\mathcal{E}_k,B_m[m-k]\right),
\end{equation}
where $B_m[m-k]$ denotes the unit ball $B_m\subset\R^m$ repeated $m-k$ times in the argument of the mixed volume $\MV$,  $\alpha_{m,k}:=\frac{m!}{(2\pi)^{k/2}(m-k)!\kappa_{m-k}}$ and $b_\infty$ is defined in Theorem~\ref{thm:incluchain}.
\end{theorem}

This result is to be compared with the centered case (i.e. the case where all $s_i$ are zero) which was proved by Kabluchko and Zaporozhets.

\begin{theorem}[{\cite[Theorem~1.1]{KabluchandZap} }]
    Let $0<k\leq m$ and let $Y_1,\ldots,Y_k\in\R^m$ be centered (i.e. mean $0$) Gaussian vectors with covariance matrix $\Sigma_i:=\EE\left[Y_i Y_i^t\right]=M_iM_i^t,$. Let $\widetilde{\Gamma}:=(Y_1,\ldots,Y_k)$. Then
    \begin{equation}
        \EE\sqrt{\det\left(\widetilde{\Gamma}^t\widetilde{\Gamma}\right)}=  \alpha_{m,k} \MV\left(\mathcal{E}_1,\ldots,\mathcal{E}_k,B_m[m-k]\right)
    \end{equation}
    where $\mathcal{E}_i:=M_i \left(B_m\right)$.
\end{theorem}

Note that we also have $\EE\seg{Y_i}=\tfrac{1}{\sqrt{2\pi}}\mathcal{E}_i$ and thus Lemma~\ref{lem:vitale} gives an alternative proof of~\cite[Theorem~1.1]{KabluchandZap}.
Moreover, one can then interpret Theorem~\ref{thm:randet} by saying that, for random determinants, the \emph{non--centered} Gaussian vector $X_i$ can be estimated from below by the \emph{centered} Gaussian vector $b_\infty Y_i$ and from above by $Y_i$. 

\begin{remark}
Lemma~\ref{lem:vitale} is still valid if the random vectors are degenerate (i.e almost surely contained in a hyperplane). Hence, Theorem~\ref{thm:randet} is still valid if the map $M_i$ is not surjective. However one cannot obtain all degenerate Gaussian vectors this way as this does not cover the case where the support is contained in an affine hyperplane.
\end{remark}

If $k=m$ and all the Gaussian vectors are identically distributed we obtain the following.

\begin{proposition}[Estimates and asymptotics in the \emph{iid} case]
Let $\xi_1,\ldots,\xi_m\in\R^m$ be iid standard Gaussian vectors, let $M:\R^m\to\R^m$ be an invertible linear map, $s\geq 0$ a fixed constant and define the iid Gaussian vectors $X_i:=M(se_1+\xi_i)$ for $i=1,\ldots,m$. Consider the random square matrix $\Gamma:=(X_1,\ldots,X_m)$ whose columns are the vectors $X_i$. We have
\begin{equation}
|\det M|\frac{\lambda(s)}{(2\pi)^{\frac{m}{2}}} \frac{2m!\kappa_{m-1}}{\sqrt{m}}\leq \EE|\det \Gamma|\leq |\det M|\frac{\lambda(s)}{(2\pi)^{\frac{m}{2}}} m!\kappa_m.
\end{equation}
Moreover, as $s\to+\infty$, we have:
\begin{equation}
\lim_{s\to+\infty}\frac{1}{s}\EE|\det \Gamma| =|\det M|\frac{1}{(2\pi)^{\frac{m-1}{2}}}\cdot\frac{m!\kappa_{m-1}}{\sqrt{m}}.
\end{equation}
\end{proposition}
\begin{proof}
By Lemma~\ref{lem:vitale}, we have that $\EE|\det \Gamma|=m!\vol_m(M(G(s)))=m!|\det M|\vol_m(G(s))$. Then the upper boud follows from Corollary~\ref{cor:volwithlambda}, and the lower bound and the limit follow from Proposition~\ref{prop:lowvolandlim}.
\end{proof}

\begin{remark}
Note that the Gaussian vectors $X_i=M(se_1+\xi_i)$ have covariance matrix $MM^t$ and thus $|\det M|=\sqrt{\det(MM^t)}$ is a function of this covariance matrix.
\end{remark}

\section{Gaussian perturbation of a hypersurface}\label{sec:GRF}
In this section we apply the previous results to study the zero sets of \emph{non--centered} Gaussian random fields. 
Let us recall that A \emph{Gaussian Random Field} (GRF) on a manifold $M$ is a random function $X: M\to \R$, such that for all finite collection of points $p_1,\ldots,p_k\in M$, the random vector $(X(p_1),\ldots,X(p_k))\in\R^k$ is a Gaussian vector. In the following, we will only consider GRFs that are almost surely smooth. We say that a GRF is \emph{centered} if for every $p\in M$, $\EE X(p)=0$. Moreover, we will assume that the GRF is non--degenerate in the following sense. 

\begin{definition}\label{def:nondegGRF}
    A centered GRF $Y:M\to\R$ on a manifold $M$ is said to be \emph{non--degenerate} if for every $p\in M$, we have $\det \left(\EE [(d_pY) (d_pY)^T]\right)\neq 0$ where $d_pY\in T^*_pM$ is the differential of $Y$ at $p$.
    The Riemannian metric thus defined will be called the metric \emph{induced} by $Y$ and denoted $g_Y$, see \cite[Chapter~12]{AT}. Concretely it is given for all $v,w\in T_pM$ by
    \begin{equation}
        g_Y(v,w):=\EE\left[d_pY(v)d_pY(w)\right].
    \end{equation}
    A GRF $X$ is said to be \emph{non--degenerate} if the centered GRF $Y:=X-\EE X$ is non--degenerate.
\end{definition}

Note that the random vector $d_pY\in T^*_pM$ is a \emph{standard Gaussian vector} for the Euclidean structure given by $g_Y$ (or more precisely its dual metric on $T^*M$), i.e. if we identify $T^*_pM$ with $\R^m$ using a basis of $T_pM$ that is orthonormal for $g_Y$ then $d_pY$ becomes a standard Gaussian vector. Hence the following definition.

\begin{definition}
    Let $(M,g)$ be a Riemannian manifold, a GRF $Y\in C^\infty(M)$ is called a \emph{standard GRF} if for all $p\in M$, we have $\EE Y(p)=0$, $\EE [Y(p)^2]=1$ and $g_{Y,p}=g_p$, where $g_{Y,p}$ is the Riemannian metric induced by $Y$ at the point $p$.
\end{definition}

\begin{remark}\label{rk:allStandard}
    The existence of a standard GRF impose no restriction on the Riemannian metrig $g$. In other words, for every Riemannian manifold $(M,g)$, there is a centered GRF $Y$ such that $g_Y=g$, see \cite[Remark~6.4]{ZonSec} or \cite{AT}.{ Note that in \cite{ZonSec} standard GRF are called \emph{normal}, here we chose to call it \emph{standard} to emphasize the analogy with the standard Gaussian vectors of the previous sections.}
\end{remark}

We use the framework {developed} in \cite{ZonSec} where the author, together with Michele Stecconi, study the zero set of a certain class of random fields (that they call zKROK fields) that contain Gaussian Random Fields \cite[Proposition~4.11]{ZonSec}. We will now reformulate the main results of \cite{ZonSec} in our context of GRFs. 

Given a (non--degenerate) GRF $X$ on a manifold $M$ and a point $p\in M$ one can consider the \emph{conditionned} GRF $(X|X(p)=0)$, see \cite[Section~4.2]{ZonSec}. Considering the differential at a point $p\in M$, one then builds the random vector $(d_pX|X(p)=0)\in T^*_pM$. Using the Vitale construction, one associates to a GRF a family of zonoids in the cotangent space, see~\cite[Definition 5.1]{ZonSec}.

\begin{definition}[Zonoid section]\label{def:zonsec}
Let $X:M\to\R$ be a GRF. For all $p\in M$, we define the following zonoid in $T^*_pM$:
\begin{equation}
	\zeta_X(p):=\rho_{X(p)}(0)\EE\left[\seg{\dd_p X} | X(p)=0 \right],
\end{equation}
where $\rho_{X(p)}(0)$ is the density of the random variable $X(p)\in\R$ at $0$ and $\EE\left[\seg{\dd_p X} | X(p)=0 \right]$ is the Vitale zonoid associated to the random vector $(\dd_pX|X(p)=0)$. We call $\zeta_X$ the \emph{zonoid section} associated to the random field $X$.
\end{definition}

\begin{remark}
    The zonoid section just defined is slightly different than in \cite[Definition 5.1]{ZonSec} where $\seg{\dd_p X}$ is replaced by $[0,d_pX]$. The zonoid section defined by Definition~\ref{def:zonsec}, is called the \emph{centered} zonoid section in \cite{ZonSec} and is denoted by $\seg{\zeta_X}$. It is a translate of what they call the \emph{zonoid section}. Nevertheless, this will not affect the results we present here since we only consider translation invariant quantities such as volume or mixed volume. 
\end{remark}

The main result of \cite{ZonSec} is that the zonoid section computes the expected volume of the zero set of random fields. In our context of (scalar) GRF, we will use the following special case, which is \cite[Corollary 7.2]{ZonSec}.

\begin{proposition}[Expected number of points and mixed volume]\label{prop:MatStekmain}
	Let $M$ be a Riemannian manifold of dimension $m$, let $X_1,\ldots,X_m:M\to\R$ be independent GRF and let $Z_i:=X_i^{-1}(0)$, $i=1,\ldots,m$. We have, for all $U\subset M$ open:
	\begin{equation}
	\EE\#(Z_1\cap\cdots\cap Z_m\cap U)=m!\int_U \MV(\zeta_{X_1}(p),\ldots,\zeta_{X_m}(p))\, \dd M(p),
	\end{equation}
	where $\dd M(p)$ denotes the integration with respect to the volume form on $M$ and $\MV$ is the \emph{mixed volume} (see beginning of Section~\ref{sec:randet}).
\end{proposition}

This is again to be compared with the result of Kabluchko and Zaporozhets \cite[Theorem~1.5]{KabluchandZap} that deals with the case where the GRF are all centered, in which case the zonoids $\zeta_{X_i}(p)$ are ellipsoids.
For a general GRF $X$, the random vector $(d_pX|X(p)=0)$ is a Gaussian vector, and thus the zonoid section $\zeta_X(p)$ is, by definition, a Gaussian zonoid.

For the rest of this section, we fix a Riemannian manifold $M$ of dimension $m$ and we define a GRF which will be our main object of study.

\begin{definition}\label{def:Xtau}
Let $Y\in C^\infty(M)$ be a standard GRF. Let $\varphi\in C^\infty(M)$ be a fixed smooth function and let $\tau\geq 0$. We define the following GRF.
\begin{equation}
	X_\tau:=\varphi+\tau Y.
\end{equation}
Moreover, we write $Z_\tau:=X_\tau^{-1}(0)$.
\end{definition}

If $0$ is a regular value of $\varphi$, i.e. $d_p\varphi\neq 0$ $\forall p\in Z_0$, then $Z_\tau$ is a one parameter family of random perturbation of the hypersurface $Z_0=\varphi^{-1}(0)$. We expect this family to be more \emph{diffuse} on $M$ as $\tau\to+\infty$ and to \emph{concentrate} near $Z_0$ as $\tau\to 0$. We will show how to make these statements precise using the zonoid section and Proposition~\ref{prop:MatStekmain}.

We can compute the zonoid section explicitely in terms of Gaussian zonoids. For this purpose, it is convenient to introduce a basis-independent version of the zonoid $G(s)$ from the previous sections.

\begin{definition}
    Let $V$ be an euclidean vector space, let $v\in V$ and let $\xi\in V$ be a standard Gaussian vector. We write
    \begin{equation}
        \vG(v):=\EE\seg{v+\xi}.
    \end{equation}
\end{definition}
If we identify $V\cong\R^m$ using an orthonormal basis $e_1,\ldots, e_m$ of course we have $G(s)=\vG(se_1)$ for all $s\geq 0$.

\begin{proposition}
For all $\tau>0$, the random field $X_\tau$ defined in Definition~\ref{def:Xtau} is a non--degenerate GRF and its zonoid section is given for all $p\in M$ by
\begin{equation}
\zeta_\tau(p):=\zeta_{X_\tau}(p)=\frac{e^{\frac{-\varphi(p)^2}{2\tau^2}}}{\sqrt{2\pi}}\vG\left(\frac{\dd_p\varphi}{\tau}\right).
\end{equation}
\end{proposition}

\begin{proof}
For all $p\in M$, $X_\tau(p)\in\R$ is a Gaussian variable with mean $\varphi(p)$ and variance $\tau^2>0$ and the centered GRF $X-\varphi$ is a multiple of the standard GRF $Y$. Thus $X_\tau$ is a non-degenerate GRF and for all $p\in M$ the density at $0$ of $X_\tau(p)$ is given by:
\begin{equation}\label{eq:rhoxt}
\rho_{X_\tau(p)}(0)=\frac{e^{\frac{-\varphi(p)^2}{2\tau^2}}}{\sqrt{2\pi}\tau}.
\end{equation}
Moreover, for all $p\in M$, $Y(p)$ is independent of $\dd_pY$. Indeed, by differentiating the equation $\EE Y(p)^2=\tau^2$, we get $\EE \left[Y(p) d_pY\right]=0$. Thus $X_\tau(p)=\varphi(p)+\tau Y(p)$ is independent of $\dd_p X_\tau=\dd_p\varphi+\tau \dd_p Y$. It follows that the random vector $(\dd_p X_\tau|X(p)=0)$ has the same law as $\dd_p X_\tau$. We obtain
\begin{equation}\label{eq:expdxt}
\EE\left[\seg{\dd_p X_\tau} | X_\tau(p)=0 \right]=\EE\seg{\dd_p X_\tau}=\EE\seg{\dd_p\varphi+\tau \dd_p Y}=\tau \cdot \vG\left(\frac{\dd_p\varphi}{\tau}\right),
\end{equation}
where, in the last equality, we used the fact that $\dd_pY$ is a standard Gaussian vector in $T^*_pM$ (with the Riemannian metric). The result then follows by multiplying \eqref{eq:expdxt} by \eqref{eq:rhoxt}.
\end{proof}

These zonoids satisfy the following properties. The proofs are completely straightforward and thus omitted here.

\begin{proposition}[Properties of $\zeta_\tau$]\label{prop:propofzetat}
The zonoid section $\zeta_\tau$ satisfy the following properties.
\begin{enumerate}
	\item For all $p\in M$, we have:
	\begin{equation}
	\zeta_\tau(p) \xrightarrow[\tau \to +\infty] {}\frac{1}{2\pi} B_M(p),
	\end{equation}
	where $B_M(p)\subset T^*_pM$ is the unit ball for the dual metric on $T^*_pM$.
	\item For all $p\in M\setminus Z_0$ we have:
	\begin{equation}
	\zeta_\tau(p) \xrightarrow[\tau \to 0] {}\{0\}.
	\end{equation}
	\item Let $p\in Z_0$, and let $\pi_p:T^*_pM\to T^*_pZ_0$ be the orthogonal projection. For all $\tau>0$, we have:
	\begin{equation}
	\pi_p(\zeta_\tau(p))=\frac{1}{2\pi} B_{Z_0}(p),
	\end{equation}
	where $B_{Z_0}(p)\subset T^*_pZ_0$ is the unit ball for the Riemannian metric on $Z_0$ induced from the Riemannian metric on $M$.
\end{enumerate}
\end{proposition}

\begin{remark}\label{rk:geopropzetatau}
    These properties have a geometric interpretation. Indeed the zonoid section $\zeta_\tau$ corresponds to a Finsler structure on $M$, i.e. a norm on each tanget space, given by the support function of the zonoids $\zeta_\tau(p)$. The zonoid $\zeta_\tau$ is then the dual unit ball of this metric (or equivalentely the unit ball of the dual metric on the cotangent space), see~\cite[Section~9]{ZonSec}. 
    For each $\tau>0$ we then have a \emph{Finsler geometry on $M$}. Property $(1)$ says that when $\tau$ is large this geometry tends to the Riemannian geometry that we started with. Property $(2)$ says that, when $\tau$ is small, everything outside $Z_0$ \emph{disappears}, i.e. the Finsler geometry shrinks around $Z_0$ as $\tau$ goes to zero. Finally, Property $(3)$ says that, for all $\tau>0$, this Finsler geometry restricted to $Z_0$ gives the Riemannian geometry of $Z_0$ (in particular it is independent of $\tau$). 
\end{remark}

Next, we show that, we can \emph{estimate} the GRF $X_\tau$ by some centered GRF.

\begin{theorem}[Estimates by centered GRF]\label{thm:centeredGRFestimate}
For every $\tau>0$, there is a centered GRF $\widetilde{X}_\tau$ such that for every GRFs $W_1,\ldots,W_{m-1}\in C^\infty(M)$ independents and independents of $X_\tau,\widetilde{X}_\tau$, writing $\widetilde{Z}_\tau:=\widetilde{X}_\tau^{-1}(0)$ and $\mathcal{W}:=W_1^{-1}(0)\cap\cdots\cap W_{m-2}^{-1}(0)$, we have, for every open set $U\subset M$:
\begin{equation}
b_\infty\cdot\EE\#\left(\widetilde{Z}_\tau\cap \mathcal{W}\cap U\right)\leq\EE\#\left(Z_\tau\cap \mathcal{W}\cap U\right)\leq \EE\#\left(\widetilde{Z}_\tau\cap \mathcal{W}\cap U\right),
\end{equation}
where recall the definition of $b_\infty$ in Theorem~\ref{thm:incluchain}.
\end{theorem}
\begin{proof}
Let $p\in M$ and let $e_1,\ldots,e_m$ be an orthonormal basis of $T^*_pM$ such that $d_p\varphi=\|d_p\varphi\|e_1$. We use this basis to identify $T^*_pM\cong \R^m$. By Theorem~\ref{thm:incluchain}, for every $p\in M$, we have that:
\begin{equation}
\frac{b_\infty}{2\pi}\mathcal{E}_\tau(p)\subset \zeta_\tau(p)\subset \frac{1}{2\pi}\mathcal{E}_\tau(p),
\end{equation}
where $\mathcal{E}_\tau(p)=e^{\frac{-\varphi(p)^2}{2\tau^2}}T_{s}B_M(p)$ with $s=\|\dd_p\varphi\|/\tau$ and where recall the definition of the linear map $T_s$ in \eqref{def:Ts}. Now the ellipsoid section $\mathcal{E}_\tau$ defines a Riemannian metric $\widetilde{g}_\tau$ on $M$ for every $\tau>0$ and thus, as observed in Remark~\ref{rk:allStandard}, there is a standard GRF $\widetilde{X}_\tau$ for this metric. In that case, the associated zonoid section is precisely $\zeta_{\widetilde{X}_\tau}(p)=\frac{1}{2\pi}\mathcal{E}_\tau(p)$ for every $p\in M$. The result then follows from Proposition~\ref{prop:MatStekmain} and monotonicity of the mixed volume.
\end{proof}
We note that the ellipsoid section $\frac{1}{2\pi}\mathcal{E}_\tau$ satisfies the properties of Proposition~\ref{prop:propofzetat}. {Consequently}, the Riemannian metric $\widetilde{g}_\tau$ also satisfy the geometric properties observed in Remark~\ref{rk:geopropzetatau}.

\begin{remark}\label{rk:vectGRF}
	A more refined analysis of \cite{ZonSec} (more precisely, using \cite[(7.2)]{ZonSec}), shows that the assumption that the random fields $W_1,\ldots,W_{m-1}$ are independent is not needed and one only needs that the random field $(W_1,\ldots,W_{m-1}):M\to\R^{m-1}$ is a vector valued non--degenerate GRF independent of $X_\tau$ and $\widetilde{X}_\tau$. Moreover, the Gaussian assumption is also superfluous and one could use instead the so called zKROK fields of \cite{ZonSec}.
\end{remark}

We conclude this section by an asymptotic analysis as $\tau\to 0$. In this limit, we expect the random zero set $Z_\tau$ to \emph{concentrate} near $Z_0$. As observed in Remark~\ref{rk:geopropzetatau}, Proposition~\ref{prop:propofzetat}-$(2)$ can already be interpreted as a \emph{concentration} near $Z_0$ in this limit. Nevertheless, we can give a more precise and quantitative result.


For all $r>0$ we define the following open neighbourhood of $Z_0$:
\begin{equation}
\mathcal{U}_r:=\set{p\in M}{|\varphi(p)|<r}\subset M.
\end{equation}
Moreover, for all $\tau,r>0$ we let
\begin{equation}
n_{r,\tau}:=\EE\#\left(Z^{(1)}_\tau\cap\cdots\cap Z^{(m)}_\tau\cap \mathcal{U}_r\right),
\end{equation}
where $Z^{(1)}_\tau,\ldots, Z^{(m)}_\tau$ are iid copies of $Z_\tau=X_\tau^{-1}(0)$.

\begin{theorem}[The limit $\tau\to 0$]\label{thm:concentration}
Let $M$ be a compact manifold of dimension $m$ and assume that $0$ is a regular value of $\varphi$, i.e. $d_p\varphi\neq 0$ for all $p\in Z_0$. Let $r=r(\tau)>0$ be such that $\lim_{\tau\to 0} r = 0$ and $\alpha:=\lim_{\tau\to 0}\tfrac{r}{\tau}\in[0,+\infty]$ exists. Then we have:
\begin{equation}
\lim_{\tau\to 0} n_{r,\tau}={\frac{2}{s_{m-1}}}\cdot \erf\left(\sqrt{\frac{m}{2}}\cdot \alpha\right)\cdot \vol_{m-1}(Z_0)
\end{equation}
{where $s_{m-1}$ denotes the $(m-1)$ dimensional volume of the unit sphere $S^{m-1}\subset \mathbb{R}^m$.}
\end{theorem}
\begin{proof}
By Proposition~\ref{prop:MatStekmain}, we have
\begin{equation}
n_{r,\tau}=m!\int_{\mathcal{U}_r}\vol_m(\zeta_\tau(p))\dd M(p)=\frac{m!}{(2\pi)^{\frac{m}{2}}}\int_{\mathcal{U}_r}e^{\frac{-m\varphi(p)^2}{2\tau^2}}\vol_m\left(\vG\left(\frac{\dd_p\varphi}{\tau}\right)\right)\dd M(p).
\end{equation}
Since $0$ is a regular value of $\varphi$, for $\tau$ small enough, there is no critical point of $\varphi$ in $\mathcal{U}_r$. Thus we can apply the smooth coarea formula for the function $\varphi:\mathcal{U}_r\to (-r,r)$ to get:
\begin{equation}
n_{r,\tau}=\frac{m!}{(2\pi)^{\frac{m}{2}}}\int_{-r}^re^{\frac{-m t^2}{2\tau^2}}\int_{S_t}\frac{1}{\|\dd_p\varphi\|}\vol_m\left(\vG\left(\frac{\dd_p\varphi}{\tau}\right)\right)\dd S_t(p)\dd t,
\end{equation}
where $S_t:=\varphi^{-1}(t)$. Now we apply the change of variable $u=\sqrt{\frac{m}{2}}\cdot \frac{t}{\tau}$ to obtain:
\begin{equation}
n_{r,\tau}=\frac{m!}{(2\pi)^{\frac{m}{2}}}	\sqrt{\frac{2}{m}}	\int_{-\sqrt{\frac{m}{2}}\cdot\frac{r}{\tau}}^{\sqrt{\frac{m}{2}}\cdot\frac{r}{\tau}} 	e^{-u^2}		\int_{S_{t(u)}}		\frac{\tau}{\|\dd_p\varphi\|}\vol_m\left(\vG\left(\frac{\dd_p\varphi}{\tau}\right)\right)\dd S_{t(u)}(p)\dd u,
\end{equation}
where $t(u)=\sqrt{\frac{2}{m}} \cdot \tau u$. We have that $t(u)\in[-r,r]$ and thus $t(u)\to 0$ uniformly in $u$ as $\tau\to 0$. Moreover, by Proposition~\ref{prop:lowvolandlim}, we have
$\frac{\tau}{\|\dd_p\varphi\|}\vol_m\left(\vG\left(\frac{\dd_p\varphi}{\tau}\right)\right)\xrightarrow[\tau\to 0]{}\frac{\kappa_{m-1}}{\sqrt{m}(2\pi)^{\frac{m-1}{2}}}.$
By compactness, this is also uniform in $p$.
We obtain, as $\tau\to 0$:
\begin{equation}
n_{r,\tau}\to \frac{(m-1)!}{(2\pi)^m}\sqrt{2\pi}\kappa_{m-1}\sqrt{2}\cdot 2\int_0^{\sqrt{\frac{m}{2}}\alpha}e^{-u^2}\dd u\,\vol_{m-1}(S_0).
\end{equation}
The result then follows by noting that $S_0=Z_0${, recognizing the error function and from the identity $(m-1)!\kappa_{m-1}=2\tfrac{(2\pi)^{m-1}}{s_{m-1}}$ (see for example \cite[Lemma A.4.]{handbook})}.

\end{proof}

\bibliographystyle{amsplain}
\bibliography{literature}

\providecommand{\bysame}{\leavevmode\hbox to3em{\hrulefill}\thinspace}
\providecommand{\MR}{\relax\ifhmode\unskip\space\fi MR }
\providecommand{\MRhref}[2]{%
  \href{http://www.ams.org/mathscinet-getitem?mr=#1}{#2}
}
\providecommand{\href}[2]{#2}
\begin{thebibliography}{10}

\bibitem{AT}
Robert~J. Adler and Jonathan~E. Taylor, \emph{Random fields and geometry},
  Springer Monographs in Mathematics, Springer New York, 2009.

\bibitem{erfineq}
Horst Alzer, \emph{Error function inequalities}, Advances in Computational
  Mathematics \textbf{33} (2010), 349--379.

\bibitem{ZA}
Paul Breiding, Peter Bürgisser, Antonio Lerario, and Léo Mathis, \emph{The
  zonoid algebra, generalized mixed volumes, and random determinants}, Advances
  in Mathematics \textbf{402} (2022), 108361.

\bibitem{foldednormal}
F.~C. Leone, L.~S. Nelson, and R.~B. Nottingham, \emph{The folded normal
  distribution}, Technometrics \textbf{3} (1961), no.~4, 543--550.

\bibitem{ZonSec}
L\'eo Mathis and Michele Stecconi, \emph{Expectation of a random submanifold:
  the zonoid section}, Annales Henri Lebesgue \textbf{7} (2024), 903--967 (en).

\bibitem{handbook}
Léo Mathis, \emph{The handbook of zonoid calculus}, Ph.D. thesis, Scuola
  Internazionale di Studi Avanzati, 2022.

\bibitem{molchRandSets}
I.~Molchanov, \emph{Theory of random sets}, Probability and Its Applications,
  Springer London, 2006.

\bibitem{MOLCH2}
Ilya Molchanov and Felix Nagel, \emph{Diagonal minkowski classes, zonoid
  equivalence, and stable laws}, Communications in Contemporary Mathematics
  \textbf{23} (2021), no.~02, 1950091.

\bibitem{MOLCH1}
Ilya Molchanov and Riccardo Turin, \emph{Convex bodies generated by sublinear
  expectations of random vectors}, Advances in Applied Mathematics \textbf{131}
  (2021), 102251.

\bibitem{MolchanovWespi}
Ilya Molchanov and Florian Wespi, \emph{{Convex hulls of Lévy processes}},
  Electronic Communications in Probability \textbf{21} (2016), no.~none, 1 --
  11.

\bibitem{MoslerLiftzon}
Karl Mosler, \emph{Multivariate dispersion, central regions and depth. the lift
  zonoid approach}, vol. 165, 01 2002.

\bibitem{bible}
Rolf Schneider, \emph{Convex bodies: the {B}runn-{M}inkowski theory}, expanded
  ed., Encyclopedia of Mathematics and its Applications, vol. 151, Cambridge
  University Press, Cambridge, 2014. \MR{3155183}

\bibitem{popeRandzon}
\bysame, \emph{Random zonotopes and valuations}, Discrete {\&} Computational
  Geometry (2023).

\bibitem{Vitale}
Richard~A. Vitale, \emph{Expected absolute random determinants and zonoids},
  Ann. Appl. Probab. \textbf{1} (1991), no.~2, 293--300. \MR{1102321}

\bibitem{wiki:foldedNormal}
{Wikipedia contributors}, \emph{Folded normal distribution --- {W}ikipedia{,}
  the free encyclopedia}, 2022, [Online; accessed 18-January-2022].

\bibitem{KabluchandZap}
D.~Zaporozhets and Z.~Kabluchko, \emph{Random determinants, mixed volumes of
  ellipsoids, and zeros of gaussian random fields}, Journal of Mathematical
  Sciences \textbf{199} (2014), no.~2, 168--173.

\end{thebibliography}

\end{document}